\documentclass[11pt,a4paper]{article}
\usepackage{amsfonts,amsgen,amsmath,amstext,amsbsy,amsopn,amsthm,amsfonts,amssymb,amscd}
\usepackage{epsf,epsfig}
\usepackage{float}
\usepackage{ebezier,eepic}
\usepackage{color}
\usepackage{tikz}
\usepackage{multirow}
\usepackage{graphicx}
\usepackage{subfigure}
\newtheorem{thm}{Theorem}[section]

\newtheorem{prop}[thm]{Proposition}
\newtheorem{prob}[thm]{Problem}
\newtheorem{lem}[thm]{Lemma}
\newtheorem{false statement}{False statement}

\theoremstyle{definition}
\newtheorem{definition}[thm]{Definition}
\newtheorem{claim}{Claim}

\newtheorem{conj}[thm]{Conjecture}

\makeatletter \@addtoreset{equation}{section}

\def\hh{\mathcal{H}}
\def\hf{\mathcal{F}}
\def\hg{\mathcal{G}}
\def\hk{\mathcal{K}}
\def\ha{\mathcal{A}}
\def\hb{\mathcal{B}}
\def\hd{\mathcal{D}}

\def\hc{\mathcal{C}}
\def\hi{\mathcal{I}}
\def\hj{\mathcal{J}}
\begin{document}
\title{\bf\Large On the Maximum Number of Edges in Hypergraphs with Fixed Matching and Clique Number}
\date{}
\author{Peter Frankl$^1$, Erica L.L. Liu$^2$, Jian Wang$^3$\\[10pt]
$^{1}$R\'{e}nyi Institute of Mathematics\\
Re\'{a}ltanoda u. 13-15\\
H-1053 Budapest, Hungary\\[6pt]
$^{2}$Center for Applied Mathematics\\
Tianjin University\\
Tianjin 300072, P. R. China\\[6pt]
$^{3}$Department of Mathematics\\
Taiyuan University of Technology\\
Taiyuan 030024, P. R. China\\[6pt]
E-mail:  $^1$peter.frankl@gmail.com, $^2$liulingling@tju.edu.cn, $^3$wangjian01@tyut.edu.cn
}

\maketitle

\begin{abstract}
For a $k$-graph $\hf\subset \binom{[n]}{k}$, the clique number of $\hf$ is defined to be the maximum size of a subset $Q$ of $[n]$ with $\binom{Q}{k}\subset \hf$. In the present paper, we determine the maximum number of edges in a $k$-graph on $[n]$ with matching number at most $s$ and clique number at least $q$ for $n\geq 8k^2s$ and for $q \geq (s+1)k-l$, $n\leq (s+1)k+s/(3k)-l$. Two special cases that $q=(s+1)k-2$ and $k=2$ are solved completely.
\end{abstract}


\medskip

\section{Introduction}

Let $n>k\geq1$ be integers. Let $[n]=\{1,2,\ldots,n\}$ be the standard $n$ element set and ${[n]\choose k}$ be the collection of all its $k$-subsets. For a $k$-graph $\hf\subset {[n]\choose k}$, let $\nu(\hf)$ be the matching number of $\hf$, that is, the maximum number of pairwise disjoint members of $\hf$.
One of the most important open problems in extremal set theory is the following.
\begin{conj}[Erd\H{o}s Matching Conjecture \cite{E65}]
Suppose that $s$ is a positive integer, $n\geq(s+1)k$ and $\hf\subset {[n]\choose k}$ satisfies $\nu(\hf)\leq s$. Then
\begin{align}\label{1-1}
|\hf|\leq\max\left\{{n\choose k}-{n-s\choose k},{(s+1)k-1\choose k}\right\}.
\end{align}
\end{conj}
The two simple constructions showing that \eqref{1-1} is best possible (if true) are
$$\mathcal{E}(n,k,s)=\left\{E\in{[n]\choose k}\colon E\cap [s]\neq \emptyset\right\} \text{ and } {[(s+1)k-1]\choose k}.$$
Erd\H{o}s and Gallai \cite{EG} proved \eqref{1-1} for $k=2$. The $k=3$ case is settled in \cite{F12}. For general $k$ Erd\H{o}s proved that \eqref{1-1} is true and up to isomorphic $\mathcal{E}(n,k,s)$ is the only optimal family provided that $n>n_0(k,s)$. The bounds for $n_0(k,s)$ were subsequently improved by Bollob\'{a}s, Daykin and Erd\H{o}s \cite{BDE}, Huang, Loh and Sudakov \cite{HLS}. The current best bounds establish \eqref{1-1} for $n>(2s+1)k$ (\cite{F13}) and for $s>s_0$, $n>\frac{5}{3}sk$ (\cite{FK}).

The case $s=1$ has received a lot of attention. If $\nu(\hf)=1$ then $\hf$ is called intersecting.

\begin{thm}[Erd\H{o}s-Ko-Rado Theorem \cite{EKR}]
Suppose that $n\geq2k$ and $\hf\subset {[n]\choose k}$ is intersecting. Then
\begin{align*}
|\hf|\leq{n-1\choose k-1}.
\end{align*}
\end{thm}

There is an important stability result related to the Erd\H{o}s-Ko-Rado Theorem.

\begin{thm}[Hilton-Milner Theorem \cite{HM}]
Suppose that $n>2k$, $\hf\subset {[n]\choose k}$ and $\cap_{F\in \hf}F=\emptyset$. Then
\begin{align}\label{1-3}
|\hf|\leq{n-1\choose k-1}-{n-k-1\choose k-1}+1.
\end{align}
\end{thm}

The Hilton-Milner family,
$$\hh(n,k)=\left\{H\in{[n]\choose k}\colon 1\in H, H\cap[2,k+1]\neq\emptyset\right\}\cup \{[2,k+1]\}$$
shows that (\ref{1-3}) is best possible. For $k\geq 4$ and $k=2$, these are unique. For $k=3$ there is one more example (up to isomorphism):
$$\mathcal{T}(n,3)=\left\{T\in{[n]\choose 3}\colon |T\cap[3]|\geq2\right\}.$$

Let us define the covering number, $\tau(\hf)$ of a family $\hf$:
$$\tau(\hf)=\min\{|T|\colon T\cap F\neq\emptyset \text{ for all }F\in\hf\}.$$
The inequalities $\nu(\hf)\leq\tau(\hf)\leq k\nu(\hf)$ are easy to check  for $\hf\subset{[n]\choose k}$. Let us note that  $\nu(\mathcal{E}(n,k,s))=\tau(\mathcal{E}(n,k,s))=s$.

Combining $\mathcal{E}(n,k,s)$ and $\hh(n,k)$ one can define
\begin{align*}
\mathcal{B}(n,k,s)=&\left\{B\in{[n]\choose k}\colon B\cap[s-1]\neq\emptyset\right\}\cup\{[s+1,s+k]\}\\
&\cup \left\{B\in{[s,n]\choose k}\colon s\in B, B\cap[s+1,s+k]\neq\emptyset\right\}.
\end{align*}

\begin{thm}[\cite{BDE}]
Suppose that $n>2k^3s$. Let $\hf\subset{[n]\choose k}$ satisfy $\nu(\hf)=s<\tau(\hf)$. Then
\begin{align*}
|\hf|\leq|\mathcal{B}(n,k,s)|.
\end{align*}
\end{thm}

Let us introduce one more parameter, the clique number $\omega(\hf)$ for $\hf\subset {[n]\choose k}$:
$$\omega(\hf)=\max\left\{q\colon \exists Q\in{[n]\choose q}, {Q\choose k}\subset \hf\right\}.$$
Let us note that $\omega(\mathcal{E}(n,k,s))=k+s-1$, $\omega(\mathcal{B}(n,k,s))=k+s$. In particular, $\omega(\hh(n,k))=k+1$.

Li, Chen, Huang and Lih \cite{LCHL}  proposed to investigate the maximum size of an intersecting family $\hf\subset{[n]\choose k}$ satisfying $\omega(\hf)=q$, $k<q<2k$. They proposed the following natural example:
$$\mathcal{L}(n,k,q)={[q]\choose k}\cup\left\{L\in{[n]\choose k}\colon 1\in L, |L\cap [q]|> q-k\right\}.$$
The best possible bound $|\hf|\leq |\mathcal{L}(n,k,q)|$ is proved in \cite{LCHL} in the cases $q=2k-1,2k-2,2k-3$ or $n$ sufficiently large with respect to $q$.

\begin{thm}[\cite{F19}]\label{thm-f19}
Let $n>2k$, $k<q<2k$. Suppose that $\hf\subset {[n]\choose k}$ satisfies $\nu(\hf)=1$, $\omega(\hf)\geq q$. Then
\begin{align*}
|\hf|\leq |\mathcal{L}(n,k,q)|.
\end{align*}
\end{thm}

In the present paper, we investigate the corresponding problem for families $\hf\subset{[n]\choose k}$ with $\nu(\hf)\geq2$.

\begin{prob}\label{pr}
Let $k, s$ be positive integers, $n\geq(s+1)k$, $q>k$. Determine or estimate
$$m(n,q,k,s):=\max\left\{|\hf|\colon \hf\subset{[n]\choose k},\nu(\hf)\leq s, \omega(\hf)\geq q\right\}.$$
\end{prob}

Since $\nu({[q]\choose k})=\lfloor q/k\rfloor$, $m(n,q,k,s)=0$ unless $q<k(s+1)$. Adding a new $k$-element set to ${[(s+1)k-1]\choose k}$ increases the matching number to $s+1$. Thus,
\begin{align*}
m(n,(s+1)k-1,k,s)={(s+1)k-1\choose k}.
\end{align*}
Consequently we consider only the cases $k<q<(s+1)k$.

The following operation, called shifting, was invented by Erd\H{o}s, Ko and Rado \cite{EKR}. Let $1\leq i< j\leq n$, $\hf\subset{[n]\choose k}$. Define
$$S_{ij}(\hf)=\{S_{ij}(F)\colon F\in\hf\}$$
where
$$S_{ij}(F)=\left\{
                \begin{array}{ll}
                  (F\setminus\{j\})\cup\{i\}, & j\in F, i\notin F \text{ and } (F\setminus\{j\})\cup\{i\}\notin \hf; \\
                  F, & \hbox{otherwise.}
                \end{array}
              \right.
$$
It is well known (cf. \cite{F87}) that shifting does not increase the matching number. It is clear that it does not decrease the clique number.

Let $(a_1, \ldots, a_k)$ denote a $k$-set with its elements ordered increasingly, i.e., $a_1<\cdots <a_k$. Define the shifting partial order $\prec$ by setting $(a_1, \ldots, a_k)\prec(b_1, \ldots, b_k)$ iff $a_l\leq b_l$ for all $1\leq l\leq k$.

\begin{definition}
The family $\hf\subset {[n]\choose k}$ is called shifted if $G\prec F$ and $F\in\hf$ always imply $G\in \hf$.
\end{definition}

\begin{prob}
Let $k, s$ be positive integers, $n\geq(s+1)k$, $(s+1)k>q>k$. Determine or estimate the function
$$m^*(n,q,k,s):=\max\left\{|\hf|\colon \hf\subset{[n]\choose k}\text{ is shifted }, \nu(\hf)=s, \omega(\hf)=q\right\}.$$
\end{prob}

\begin{prop}\label{prop-2}
Suppose that $\hf \subset \binom{[n]}{k}$, $\hf$ is shifted and $\nu(\hf)\leq s$, then (i), (ii) or (iii) holds.

(i) $\hf \subset \mathcal{E}(n,k,s)$;

(ii) $s+k \leq \omega(\hf)<sk+k-1$;

(iii) $\hf = \binom{[sk+k-1]}{k}$.
\end{prop}

\begin{proof}
If $\hf\subset\mathcal{E}(n,k,s)$, then (i) holds. Otherwise $F\cap [s] =\emptyset$ must hold for some $F\in \hf$. Then $(s+1,s+2,\ldots,s+k) \prec F$ implies $\{s+1,\ldots,s+k\}\in \hf$. Again by shiftedness $\binom{[s+k]}{k}\subset \hf$ showing that $s+k \leq \omega(\hf)$.

Now (ii) holds unless  $\omega(\hf)\geq sk+k-1$. Then by shiftedness $\binom{[sk+k-1]}{k}\subset \hf$. Should there exist some $F_0 \in \hf \setminus \binom{[sk+k-1]}{k}$, $|F_0\cap [sk+k-1]|\leq k-1$ follows. Then we can find pairwise disjoint $k$-sets $F_1,\ldots,F_s \subset [sk+k-1]\setminus F_0$ and get a contradiction with $\nu(F) =s$.
\end{proof}

Since repeated application of shifting  eventually produces a shifted family, in view of the above discussion, we have
\begin{align}\label{1-7}
m(n,q,k,s)=\max_{q\leq t<(s+1)k}m^*(n,t,k,s).
\end{align}
One aim of the present paper is to determine $m^*(n,t,k,s)$ for $n$ sufficiently large, e.g., $n\geq 8k^2s$. Unless otherwise stated, from now on $\hf$ always denotes a shifted family satisfying $\hf\subset {[n]\choose k}$, $\nu(\hf)=s$, $\omega(\hf)=q$ and $|\hf|=m^*(n,q,k,s)$. Since $\hf$ is shifted, ${[q]\choose k}\subset \hf$.

Let us recall the following common notations:
$$\hf(i)=\{F\setminus\{i\}\colon i\in F\in \hf\}, \qquad \hf(\bar{i})= \{F\in\hf: i\notin F\}.$$
Note that $|\hf|=|\hf(i)|+|\hf(\bar{i})|$. For $V\subset [n]$ and $q< n$, we also use
\[
\hf(V) =\{F\setminus V\colon V\subset F\in \hf\}, \qquad \hf(\overline{V}) =\{F\in F\colon F\cap V=\emptyset\},
\]
and
\[
\hf(V,[q])=\{F\setminus V\colon V\subset F\in \hf, F\setminus V \subset [q]\}.
\]

\begin{definition}
If $\nu(\hf(\bar{1}))=s-1$ then $\hf$ is called reducible.
\end{definition}

\begin{lem}\label{recur}
If $\hf$ is reducible,  then
\begin{align}\label{1-8}
m^*(n,q,k,s)= {n-1\choose k-1}+m^*(n-1,q-1,k, s-1).
\end{align}
\end{lem}

\begin{proof}
By the maximality of $\hf$, it follows that all $k$-sets containing $1$ are in $\hf$. Note that this implies that for every clique $S$ in $\hf$, $S\cup \{1\}$ is also a clique. Hence $\hf(\bar{1})$ has matching number $s-1$, clique number $q-1$ and the maximum number of edges. It follows that
\[
|\hf(\bar{1})|  = m^*(n-1,q-1,k,s-1)
\]
and the lemma follows.
\end{proof}

\begin{definition}
 If $sk+1\leq q\leq sk+k-1$, set $r=q-sk$ and define
\[
\ha(n,q,k,s) = \binom{[q]}{k}\cup \left\{A\in \binom{[n]}{k}\colon 1\in A,\ |A\cap [2,q]|\geq r\right\}.
\]
 If $s+k\leq q\leq sk$, let $p$ be the integer satisfying $(s-p)k+p+1\leq q\leq (s-p)k+p+k-1$ and  $r=q-p-(s-p)k$. Define a family on $[n]$ by
\begin{align*}
\ha(n,q,k,s) =\binom{[q]}{k}&\cup \left\{A\in \binom{[p+1,n]}{k}\colon p+1\in A,\ |A\cap [p+2,q]|\geq r\right\} \\ &\qquad\qquad \cup\left\{A\in \binom{[n]}{k}\colon A\cap[p]\neq \emptyset \right\}.
\end{align*}
If $q=s+k-1$, define
\[
\ha(n,q,k,s) =  \mathcal{E}(n,k,s).
\]
\end{definition}

%

It is easy to check that $\ha(n,q,k,s)$ is a shifted $k$-graph with matching number $s$ and clique number $q$. Therefore,
\[
m^*(n,q,k,s) \geq |\ha(n,q,k,s)|=\binom{n}{k}-\binom{n-p}{k}+\binom{q-p}{k}+\sum_{i=r+1}^{k-1} \binom{q-p-1}{i-1}\binom{n-q}{k-i}.
\]

In the present paper, we first prove the following result.

\begin{thm}\label{main-1}
For $s+k\leq q< sk+k-1$ and $n\geq 8k^2s$,
\[
m^*(n,q,k,s)  =  |\ha(n,q,k,s)|.
\]
\end{thm}

To answer Problem \ref{pr}, let us make the following conjecture.
\begin{conj}\label{conj-2}
Let $k, s$ be positive integers, $n\geq(s+1)k$, $s+k-1\leq  q\leq sk+k-1$.
\begin{align}\label{eqconj-2}
m(n,q,k,s)= \max\left\{|\ha(n,q,k,s)|, \binom{sk+k-1}{k}\right\}.
\end{align}
\end{conj}

By \eqref{1-7} and Theorem \ref{main-1}, we prove \eqref{eqconj-2} for $n\geq 8k^2s$.

\begin{thm}\label{main-2}
If $s+k\leq q\leq sk+k-1$ and $n\geq 8k^2s$,
then
\[
m(n,q,k,s) =|\ha(n,q,k,s)|.
\]
If $k\leq q\leq s+k-1$ and $n\geq 8k^2s$, then
\[
m(n,q,k,s) =\binom{n}{k}-\binom{n-s}{k}.
\]
\end{thm}

In \cite{F17}, the first author showed that for a small $\varepsilon$ ($\varepsilon =\varepsilon(k)$) and
$(s+1)k\leq n<(s+1)(k+\varepsilon)$ the Erd\H{o}s Matching Conjecture is true and
\begin{align*}
|\hf| \leq \binom{sk+k-1}{k}.
\end{align*}
In the present paper, we prove a similar result about  Conjecture \ref{conj-2}.

\begin{thm}\label{main-4}
Let $l<\frac{s}{3k}$ and $(s+1)k\leq n\leq (s+1)k+\frac{s}{3k}-l$. Then
\[
m(n,(s+1)k-l,k,s) = \binom{sk+k-1}{k}.
\]
\end{thm}

We also prove a general result about cross-intersecting families, which may be of independent interests.

\begin{thm}\label{prop6-1}
Let $n,k,l,s$ be positive integers, $t\geq 0$. Let $\mathcal{A}\subset{[n]\choose k}$ and $\mathcal{B}\subset {[n]\choose l}$. Suppose that $\nu(\mathcal{A})\leq s$ and $\mathcal{B}$ is $t$-intersecting (for $t=0$ the condition is void). Suppose further that $n\geq\max\{k+l,(2s+1)k,(l-t+1)(t+1)\}$, $s\geq t$ and $\mathcal{A}$, $\mathcal{B}$ are cross-intersecting, $\beta>0$ is a constant. Then
\begin{align}\label{6-1}
|\mathcal{A}|+\beta|\mathcal{B}|\leq \max_{t\leq i\leq s}\left\{{n\choose k}-{n-i\choose k}+\beta{n-i\choose l-i}\right\}.
\end{align}
\end{thm}

Note that the families giving equality in \eqref{6-1} are:
\[\mathcal{A}_i=\left\{A\in{[n]\choose k}\colon A\cap[i]\neq\emptyset\right\}, \qquad \mathcal{B}_i=\left\{B\in{[n]\choose l}\colon [i]\subset B\right\}.\]

Using Theorem \ref{prop6-1}, we give a proof of \eqref{eqconj-2} for $q=(s+1)k-2$.

\begin{thm}\label{specialcase-2}
For $n\geq(s+1)k$,
\[
m(n,sk+k-2,k,s)= \max\left\{{sk+k-1\choose k}, {sk+k-2\choose k}+{sk+k-3\choose k-2}(n-q)\right\}.
\]
\end{thm}

For $k=2$, Conjecture \ref{conj-2} is confirmed.

\begin{thm}\label{specialcase-1}
Let $s\geq 2$, $s+1\leq q\leq 2s+1$ and $n\geq 2s+2$. Then
\[
m(n,q,2,s) = \max \left\{ \binom{2s+1}{2},\binom{q}{2}+(2s+1-q)(n-q)\right\}.
\]
\end{thm}

\section{A cross-intersecting theorem for direct products}

The first author \cite{F96} proved an Erd\H{o}s-Ko-Rado theorem for direct products by applying the cyclic permutation method of Katona \cite{K72}. In this section, we extend it to a cross-intersecting theorem for direct products, which is a principal tool in proving Theorem \ref{main-1}.

Suppose that $1\leq l< k$ and $X = X_1\cup X_2$ with $|X_i| = n_i$ for $i=1,2$. Define
\[
\hh(n_1,n_2,k,l) =\left\{F\in \binom{X}{k}\colon |F\cap X_1| =l \mbox{ and } |F\cap X_2| =k-l \right\}.
\]

\begin{thm}[\cite{F96}] \label{directEKR}
 Suppose that $n_1\geq 2l$, $n_2\geq 2(k-l)$ and $\hf\subset \hh(n_1,n_2,k,l)$ is intersecting. Then
\[
|\hf|\leq \max\left\{\binom{n_1-1}{l-1}\binom{n_2}{k-l},\binom{n_1}{l}\binom{n_2-1}{k-l-1}\right\}.
\]
\end{thm}
 Let $\sigma=(x_0,\ldots,x_{m-1})$ be a cyclic permutation on $[m]$, $1\leq l<m$ an integer. Define
 \[
 \hc(\sigma,l) =\left\{\{x_i,x_{i+1},\ldots,x_{i+l-1}\}\colon i=0,1,\ldots,m-1\right\},
 \]
 reducing the subscripts modulo $m$.

\begin{prop}[\cite{F96}]\label{prop-1}
Let $\sigma$ be a cyclic permutation on $[m]$. If $\hb\subset \hc(\sigma, b)$ and $\hd\subset \hc(\sigma, d)$ are cross-intersecting and $b+d\leq m$, then the following hold:

(i) $|\hb|+|\hd| \leq m$,

(ii) $|\hb|+|\hd| \leq b+d$ if both $\hb$ and $\hd$ are non-empty.
\end{prop}

\begin{lem}[\cite{F96}]\label{cross-interval}
Let $\sigma$ be a cyclic permutation on $X_1$ and $\pi$ be a cyclic permutation on $X_2$. If $\hk \subset \hc(\sigma,l)\times \hc(\pi,k-l)$ is an intersecting family then
\[
|\hk|\leq \max\{ln_2,(k-l)n_1\}.
\]
\end{lem}

We introduce some notations. Let $\hk \subset \hc(\sigma,l)\times \hc(\pi,k-l)$.
For any $I\in \hc(\sigma,l)$, let
\[
N(I,\hk) = \{J\in \hc(\pi,k-l)\colon I\cup J \in \hk\}
\]
and $\deg(I,\hk) =|N(I,\hk)|$.

\begin{thm}\label{cross-direct}
Let $1\leq l<l'\leq k-1$, $n_2\geq 4kn_1$ and $n_1\geq l+l'$. Let  $\hf\subset \hh(n_1,n_2,k,l)$ and $\hf'\subset \hh(n_1,n_2,k,l')$. Suppose that $\hf$ is intersecting and $\nu(\hf')\leq s$, moreover $\hf$ and $\hf'$ are cross-intersecting. Then we have
\[
|\hf| +|\hf'|\leq \max\left\{\binom{n_1-1}{l-1}\binom{n_2}{k-l}+\binom{n_1-1}{l'-1}\binom{n_2}{k-l'}, 2s\binom{n_1-1}{l'-1}\binom{n_2}{k-l'}\right\}.
\]
\end{thm}

\begin{proof}
Let $\sigma$ be a cyclic permutation on $X_1$ and $\pi$ be a cyclic permutation on $X_2$. Define the {\it restriction} of $\hf$ on $\sigma$ and $\pi$ as follows:
\[
\hf\big|_{\sigma,\pi}= \left\{F\in \hf\colon F\cap X_1\in \hc(\sigma,l)\mbox{ and } F\cap X_2\in \hc(\pi,k-l)\right\}
\]
and
\[
\hf'\big|_{\sigma,\pi}= \left\{F\in \hf'\colon F\cap X_1\in \hc(\sigma,l')\mbox{ and } F\cap X_2\in \hc(\pi,k-l')\right\}.
\]

We further define
\[
\hi= \left\{I\in \hc(\sigma,l) \colon \deg(I,\hf|_{\sigma,\pi})> k-l+k-l' \right\},
\]
\[
\hj= \left\{J\in \hc(\sigma,l) \colon \deg(J,\hf|_{\sigma,\pi})\geq  1 \right\},
\]
\[
\hi' = \left\{I'\in \hc(\sigma,l') \colon \deg(I',\hf'|_{\sigma,\pi})> k-l+k-l' \right\},
\]
and
\[
\hj' = \left\{J'\in \hc(\sigma,l') \colon \deg(J',\hf'|_{\sigma,\pi})\geq  1\right\}.
\]

\begin{claim}\label{claim-1}
If $\hf\big|_{\sigma,\pi}\neq \emptyset$, then
\[
|\hf\big|_{\sigma,\pi}|+|\hf'\big|_{\sigma,\pi}|\leq (l+l') n_2.
\]
\end{claim}

\begin{proof}
Suppose to the contrary that
\begin{align}\label{eqn2-1}
|\hf\big|_{\sigma,\pi}|+|\hf'\big|_{\sigma,\pi}|> (l+l') n_2.
\end{align}
Then it follows that $|\hi|+|\hi'|\geq l+l'$. Otherwise, since $n_2\geq 4kn_1$, we have
\begin{align*}
|\hf\big|_{\sigma,\pi}|+|\hf'\big|_{\sigma,\pi}|& \leq (l+l'-1)n_2 +(2n_1-l-l'+1)(k-l+k-l')\\
&=(l+l')n_2-(n_2-(2n_1-l-l'+1)(k-l+k-l'))\\
&\leq (l+l')n_2,
\end{align*}
which contradicts our assumption. Now we distinguish four cases.

Case 1. $\hi=\emptyset$.

We claim that $\hj=\emptyset$.  Otherwise, Proposition \ref{prop-1} (ii) implies that there are $I'\in \hi'$ and $J\in \hj$ with $I'\cap J=\emptyset$. Since $\deg(I',\hf'\big|_{\sigma,\pi})> k-l+k-l'$ and $\deg(J,\hf\big|_{\sigma,\pi})\geq 1$, Proposition \ref{prop-1} (ii) implies that $N(I',\hf'\big|_{\sigma,\pi})$ and $N(J,\hf\big|_{\sigma,\pi})$ are not cross-intersecting, which contradicts the fact that $\hf$ and $\hf'$ are cross-intersecting. Thus, $\hj=\emptyset$. It follows that $\hf\big|_{\sigma,\pi}= \emptyset$, which contradicts the condition.

Case 2.  $\hi'=\emptyset$.

As in Case 1, we infer $\hj'=\emptyset$. It follows that $\hf'\big|_{\sigma,\pi}= \emptyset$. Since $\hf\big|_{\sigma,\pi}$ is intersecting, by Lemma \ref{cross-interval} we get
\[
|\hf\big|_{\sigma,\pi}| \leq \max\{ln_2,(k-l)n_1\} =ln_2,
\]
which contradicts \eqref{eqn2-1}.

Case 3. $\hi\neq \emptyset$, $\hi'\neq \emptyset$ and $|\hi|+|\hi'|\geq l+l'+1$.

Proposition \ref{prop-1} (ii) implies that there are $I\in \hi$ and $I'\in \hi'$ with $I\cap I'=\emptyset$. Since $\deg(I,\hf\big|_{\sigma,\pi})> k-l+k-l'$ and $\deg(I',\hf'\big|_{\sigma,\pi})> k-l+k-l'$, Proposition \ref{prop-1} (ii) implies that $N(I,\hf\big|_{\sigma,\pi})$ and $N(I',\hf'\big|_{\sigma,\pi})$ are not cross-intersecting, which contradicts the fact that $\hf$ and $\hf'$ are cross-intersecting.

Case 4.  $\hi\neq \emptyset$, $\hi'\neq \emptyset$ and $|\hi|+|\hi'|= l+l'$.

Then we claim that $\hj=\hi$ and $\hj'=\hi'$. Otherwise, without loss of generality, we may assume that $A\in \hj\setminus\hi$. Proposition \ref{prop-1} (ii) implies that there are $I\in \hi\cup \{A\}$ and $I'\in \hi'$ with $I\cap I'=\emptyset$. Since $\deg(I,\hf\big|_{\sigma,\pi})\geq 1$ and $\deg(I',\hf'\big|_{\sigma,\pi})> k-l+k-l'$, Proposition \ref{prop-1} (ii) implies that $N(I,\hf\big|_{\sigma,\pi})$ and $N(I',\hf'\big|_{\sigma,\pi})$ are not cross-intersecting, which contradicts the fact that $\hf$ and $\hf'$ are cross-intersecting. Thus, $\hj=\hi$, $\hj'=\hi'$ and it follows that
\[
|\hf\big|_{\sigma,\pi}|+|\hf'\big|_{\sigma,\pi}|\leq (l+l') n_2,
\]
which contradicts \eqref{eqn2-1}.
\end{proof}

\begin{claim} \label{claim-new1}
\begin{align}\label{eqn2-2}
|\hf'\big|_{\sigma,\pi}| \leq 2sl'n_2.
\end{align}
\end{claim}
\begin{proof}
Note that $\nu(\hf')\leq s$. If $n_1\leq 2sl'$, then
\[
|\hf'\big|_{\sigma,\pi}| \leq n_1n_2 \leq 2sl'n_2.
\]
Thus, we may assume that $n_1>2sl'$.
Let
\[
\hk= \left\{I\in \hc(\sigma,l') \colon \deg(I,\hf'|_{\sigma,\pi})> 2s(k-l') \right\}.
\]
Let us show that $|\hk| < (s+1)l'$. Suppose for contradiction that $|\hk| \geq (s+1)l'$. Note that $\sigma=(x_0,\ldots,x_{m-1})$. We order the members of $\hk$ as $K_1,K_2,\ldots,K_{(s+1)l'},\ldots$ with respect to the permutation $x_0\ldots x_{m-1}$. Then $\{K_1,K_{l'+1},\ldots,K_{sl'+1}\}$ is a matching of size $s+1$. Since $\deg(K_{il'+1},\hf'|_{\sigma,\pi})> 2s(k-l')$ for each $i=0,1,\ldots,s$, we can easily find pairwise disjoint sets $E_1\in N(K_{il'+1},\hf'|_{\sigma,\pi})$, $\ldots$, $E_{s+1}\in N(K_{sl'+1},\hf'|_{\sigma,\pi})$. Then $\{K_1\cup E_1$, $K_{l'+1}\cup E_2$, $\ldots$, $K_{sl'+1}\cup E_{s+1} \}$ form a matching of size $s+1$ in $\hf'$, a contradiction. Thus $|\hk| < (s+1)l'$. We infer:
\begin{align*}
|\hf'\big|_{\sigma,\pi}| < (s+1)l' n_2+ (n_1-(s+1)l') 2s(k-l').
\end{align*}
Since $n_2\geq 4kn_1$ and $s\geq 2$, $n_12sk\leq (s-1)n_2$. Thus \eqref{eqn2-2} follows.
\end{proof}

It is easy to see that
\[
|\hf| = \alpha\sum_{\sigma \in {\rm Cyc}(X_1) \atop \pi\in {\rm Cyc}(X_2)}  |\hf\big|_{\sigma,\pi}|
\]
and
\[
|\hf'| = \beta\sum_{\sigma \in {\rm Cyc}(X_1)\atop\pi\in {\rm Cyc}(X_2)}  |\hf'\big|_{\sigma,\pi}|,
\]
where
\[
\alpha =\frac{1}{l!(n_1-l)!(k-l)!(n_2-k+l)!},\ \beta=\frac{1}{l'!(n_1-l')!(k-l')!(n_2-k+l')!},
\]
and ${\rm Cyc}(X_i)$ represents the set of all cyclic permutations on $X_i$ for each $i=1,2$.
Thus,
\[
|\hf|+|\hf'| =\sum_{\sigma \in {\rm Cyc}(X_1)\atop\pi\in {\rm Cyc}(X_2)} \left(\alpha|\hf\big|_{\sigma,\pi}|+\beta|\hf'\big|_{\sigma,\pi}|\right).
\]

\begin{claim}\label{claim2-1}
$\alpha >\beta$.
\end{claim}

\begin{proof}
For $i\in \{1,2,\ldots,k-1\}$, define
\[
f(i) =\frac{1}{i!(n_1-i)!(k-i)!(n_2-k+i)!}.
\]
Then for $1\leq i\leq k-2$, we have
\[
\frac{f(i)}{f(i+1)}=  \frac{(i+1) (n_2-k+i+1)}{(n_1-i)(k-i)} \geq \frac{n_2-k}{kn_1}.
\]
Since $n_2\geq 4kn_1$ implies $\frac{n_2-k}{kn_1}>1$, it follows that $f(i)>f(i+1)$.
Thus, $\alpha = f(l)>f(l+1)>\cdots>f(l')=\beta$.
\end{proof}

If $\hf\big|_{\sigma,\pi}\neq \emptyset$, by Claim \ref{claim-1} we have
\begin{align*}
\alpha|\hf\big|_{\sigma,\pi}|+\beta|\hf'\big|_{\sigma,\pi}| \leq \alpha|\hf\big|_{\sigma,\pi}|+\beta\left((l+l')n_2-|\hf\big|_{\sigma,\pi}|\right) = (\alpha-\beta)|\hf\big|_{\sigma,\pi}|+ \beta(l+l')n_2.
\end{align*}
Since Lemma \ref{cross-interval} implies $|\hf\big|_{\sigma,\pi}|\leq ln_2$, by Claim \ref{claim2-1}
\begin{align}\label{eq2-1}
\alpha|\hf\big|_{\sigma,\pi}|+\beta|\hf'\big|_{\sigma,\pi}|\leq (\alpha-\beta)ln_2+ \beta(l+l')n_2=\alpha ln_2+\beta l'n_2.
\end{align}

If $\hf\big|_{\sigma,\pi}= \emptyset$, then by Claim \ref{claim-new1} we have
\begin{align}\label{eq2-2}
\alpha|\hf\big|_{\sigma,\pi}|+\beta|\hf'\big|_{\sigma,\pi}| \leq 2\beta sl' n_2.
\end{align}
By \eqref{eq2-1} and \eqref{eq2-2},
\begin{align*}
\alpha|\hf\big|_{\sigma,\pi}|+\beta|\hf'\big|_{\sigma,\pi}|\leq  \max\left\{\alpha ln_2+\beta l'n_2, 2\beta sl' n_2\right\}.
\end{align*}
Consequently,
\begin{align*}
|\hf|+|\hf'| &\leq \sum_{\sigma \in {\rm Cyc}(X_1)\atop\pi\in {\rm Cyc}(X_2)} \max\left\{\alpha ln_2+\beta l'n_2, 2\beta sl' n_2\right\}\\
&= \max\left\{\binom{n_1-1}{l-1}\binom{n_2}{k-l}+\binom{n_1-1}{l'-1}\binom{n_2}{k-l'},2s\binom{n_1-1}{l'-1}\binom{n_2}{k-l'}\right\}.
\end{align*}
This completes the proof.
\end{proof}

\section{The extremal number for shifted hypergraphs}

In this section, we first determine $m^*(n,q,k,s)$ for $sk+1\leq q< sk+k-1$. Then using the recursion
\eqref{1-8}, we determine $m^*(n,q,k,s)$ for $s+k\leq q\leq sk$. Because of Theorems \ref{thm-f19} and \ref{specialcase-1}, in this section we always assume that $s\geq 2$ and $k\geq 3$.

\begin{lem}\label{lem3.1}
If $q=sk+r$ with $1\leq  r< k-1$, then for $n\geq 8k^2s$
\[
m^*(n,q,k,s)  =\binom{q}{k}+\sum_{i=r+1}^{k-1} \binom{q-1}{i-1}\binom{n-q}{k-i}.
\]
\end{lem}

\begin{proof}
Let $\hf$ be a maximal shifted family with $\nu(\hf)= s$ and $\omega(\hf)= q$. Since $q=sk+r$ and $\binom{[q]}{k}\subset \hf$, we have $|F\cap [q]|\geq r+1$ for any $F\in \hf$. Let $l$ be an integer with $r+1\leq l\leq k$. Define the subfamily
\[
\hf_l =\{F\in \hf\colon |F\cap [q]|=l\}.
\]
Then
\[
|\hf| =\sum_{i=r+1}^{k} |\hf_i|.
\]

\begin{claim}\label{claim-3}
If $r+1\leq l\leq l'\leq k-1$ and $l+l'\leq k+r$, then $\hf_l$ and $\hf_{l'}$ are cross-intersecting.
\end{claim}

\begin{proof}
Suppose not, let $F_1\in \hf_l$ and $F_2\in \hf_{l'}$ be disjoint. Since $l+l'\leq k+r$ implies $q-l-l'\geq (s-1)k$, we can find $s-1$ pairwise disjoint sets $F_{3},\ldots,F_{s+1}\subset [q]\setminus (F_1\cup F_2)$ and this contradicts $\nu(\hf)=s$.
\end{proof}

\begin{claim}\label{claim-5}
For $r+1\leq l'\leq k-1$, $\nu(\hf_{l'})\leq  \frac{k-r-1}{k-l'}$.
\end{claim}

\begin{proof}
Suppose that we have $\nu(\hf_{l'})=t$. Removing the altogether $l't$ vertices of the $t$ edges forming a matching in $\hf_{l'}$, at least $sk+r-l't$ vertices in $[q]$ remain.  If this number is at least $(s+1-t)k$, we get a contradiction. Thus
\[
sk+r-l't\leq (s+1-t)k-1.
\]
Rearranging gives
\[
(k-l')t\leq k-r-1.
\]
Thus,
\[
t\leq \frac{k-r-1}{k-l'}.
\]
\end{proof}

For a pair $(l,l')$ with $r+1\leq l<l'\leq k-1$ and $l+l'=k+r$, by Claim \ref{claim-3} we see that $\hf_l$ is intersecting and $\hf_l$, $\hf_{l'}$ are cross-intersecting. From Claim \ref{claim-5}, we have $\nu(\hf_l')\leq \frac{k-r-1}{k-l'}$.
Note that $n\geq 8k^2s$ implies $n-q\geq 4k^2(s+1)\geq 4kq$. By Theorem \ref{cross-direct} we have
\begin{align}\label{eqnn3-1}
|\hf_l|+|\hf_{l'}|\leq \max\left\{\binom{q-1}{l-1}\binom{n-q}{k-l}+\binom{q-1}{l'-1}\binom{n-q}{k-l'},\frac{2 (k-r-1)}{k-l'} \binom{q-1}{l'-1}\binom{n-q}{k-l'}\right\}.
\end{align}

Note that
\[
\frac{\binom{n-q}{k-l}}{\binom{n-q}{k-l'}} =\frac{(n-q-k+l+1)\cdots (n-q-k+l')}{(k-l'+1)\cdots (k-l)}\geq \left(\frac{n-q-k+l+1}{k-l}\right)^{l'-l}
\]
and
\[
\frac{\binom{q-1}{l'-1}}{\binom{q-1}{l-1}} =\frac{(q-l'+1)\cdots (q-l)}{l\cdots (l'-1)}\leq \frac{(q-l)^{l'-l}}{l'-1}.
\]
It follows that
\begin{align*}
\binom{q-1}{l-1}\binom{n-q}{k-l}&\geq  \frac{l'-1}{(q-l)^{l'-l}} \cdot \left(\frac{n-q-k+l+1}{k-l}\right)^{l'-l} \binom{q-1}{l'-1}  \binom{n-q}{k-l'}\\[5pt]
&= (l'-1) \cdot \left(\frac{n-q-k+l+1}{(k-l)(q-l)}\right)^{l'-l} \binom{q-1}{l'-1}  \binom{n-q}{k-l'}.
\end{align*}
Since $n\geq 8k^2s$ implies that  $\frac{n-q-k+l+1}{(k-l)(q-l)}\geq 2$, we have
\begin{align}\label{eqn3-7}
\binom{q-1}{l-1}\binom{n-q}{k-l}&\geq (l'-1) 2^{l'-l} \binom{q-1}{l'-1}  \binom{n-q}{k-l'}\nonumber\\[5pt]
&= \frac{(k-l')(l'-1)}{k-2r-2+l'} \cdot  2\left(\frac{2 (k-r-1)}{k-l'}-1\right)\binom{q-1}{l'-1}  \binom{n-q}{k-l'}.
\end{align}
Since $r\geq 1$ and $\frac{k+r}{2}< l'\leq k-1$, we have $(k-l')(l'-1)\geq k-2$ and $k-2r-2+l'< 2k-4$. It follows that
\begin{align}\label{eqn3-8}
\frac{(k-l')(l'-1)}{k-2r-2+l'}>\frac{1}{2}.
\end{align}
Therefore, by \eqref{eqn3-7} and \eqref{eqn3-8}
\begin{align}\label{eqn3-9}
\binom{q-1}{l-1}\binom{n-q}{k-l}>\left(\frac{2 (k-r-1)}{k-l'}-1\right)\binom{q-1}{l'-1}  \binom{n-q}{k-l'}.
\end{align}
From \eqref{eqnn3-1} and \eqref{eqn3-9}, we get
\begin{align}\label{eqn3-15}
|\hf_l|+|\hf_{l'}|\leq \binom{q-1}{l-1}\binom{n-q}{k-l}+\binom{q-1}{l'-1}\binom{n-q}{k-l'}.
\end{align}

If $k+r$ is even and $2l=k+r$, since $\hf_l$ is intersecting, applying Theorem \ref{directEKR} gives
\[
|\hf_l| \leq \max \left\{\binom{q-1}{l-1}\binom{n-q}{k-l},\binom{q}{l}\binom{n-q-1}{k-l-1}\right\}.
\]
Since $\frac{q}{l}\leq \frac{n-q}{k-l}$ is equivalent to $kq \leq ln$, for $n\geq kq/(r+1)$, we have
\begin{align}\label{eqn3-16}
|\hf_l| \leq \binom{q-1}{l-1}\binom{n-q}{k-l}.
\end{align}
Therefore, by \eqref{eqn3-15} and \eqref{eqn3-16} we arrive at
\[
|\hf| = \binom{q}{k}+\sum_{i=r+1}^{k-1} |\hf_i| \leq  \binom{q}{k}+\sum_{i=r+1}^{k-1} \binom{q-1}{i-1}\binom{n-q}{k-i}.
\]
This completes the proof.
\end{proof}

Let us recall the following two results due to the first author and Huang, Loh, Sudakov, respectively.

\begin{lem}[\cite{F87}]\label{gbEMC}
Suppose that $\hf\subset \binom{[n]}{k}$, $n\geq k(s+1)$ and $\nu(\hf)\leq s$. Then
\[
|\hf| \leq s\binom{n-1}{k-1}.
\]
\end{lem}

\begin{lem}[\cite{HLS}]\label{rainbow}
Let $\mathcal{H}_1, \ldots, \mathcal{H}_s\subset {[n]\choose k}$ be families satisfying $|\mathcal{H}_i|>(s-1){n-1\choose k-1}$ for $i=1\ldots, s$, and $n\geq sk$. Then there exist $s$ pairwise disjoint sets $H_1\in\mathcal{H}_1, \ldots, H_s\in\mathcal{H}_s$.
\end{lem}

\begin{lem}\label{reducible}
Let $\hf\subset {[n]\choose k}$ be a shifted family satisfying $\nu(\hf)=s$, $\omega(\hf)=q$ and $|\hf|=m^*(n,q,k,s)$. If $s+k\leq q\leq sk$ and $n\geq 8 k^2s$, then $\hf$ is reducible.
\end{lem}
\begin{proof}
Suppose for contradiction that $\hf$ is not reducible. We shall show that
\[
|\hf| <p\binom{n-p}{k-1}+\binom{q-p}{k}\leq m^*(n,q,k,s),
\]
which contradicts the condition that  $|\hf|=m^*(n,q,k,s)$. Recall that for $1\leq i\leq k$,
\[
\hf_i =\{F\in \hf\colon |F\cap [q]|=i\}.
\]
Then
\[
|\hf| =\sum_{i=1}^{k} |\hf_i|.
\]

Note that $p$ is the integer satisfying $(s-p)k+p+1\leq q\leq (s-p)k+p+k-1$. Since $q\geq s+k$, we have $p<s$. It follows that
\begin{align}\label{eq3-10}
q\geq p+k-1.
\end{align}

\begin{claim}\label{claim-4}
For $i=1,\ldots,k-1$, $\nu(\hf_i)<\frac{(k-1)(p+1)}{k-i}$.
\end{claim}

\begin{proof}
Suppose that we have $\nu(\hf_i)=t$. Removing the altogether $it$ vertices of the $t$ edges forming a matching in $\hf_i$, at least $q-it\geq (s-p)k+p+1-it$ vertices in $[q]$ remain.  If this number is at least $(s+1-t)k$, we get a contradiction. Thus
\[
(s-p)k+p+1-it\leq (s+1-t)k-1.
\]
Rearranging gives
\[
(k-i)t\leq (k+1)(p+1)-1.
\]
Therefore,
\[
t\leq \frac{k-1}{k-i}(p+1) -\frac{1}{k-i}<\frac{k-1}{k-i}(p+1).
\]
\end{proof}

\begin{claim}\label{claim3-1}
\begin{align}\label{eqn3-1}
|\hf_1| \leq \frac{1}{8} p\binom{n-p}{k-1}.
\end{align}
\end{claim}

\begin{proof}
Since $\hf$ is not reducible, there are $s$ pairwise disjoint sets $F_1',\ldots,F_s'$ in $\hf(\bar{1})$ (This is the only  place using the assumption that $\hf$ is not reducible). Set $T'=F_1'\cup\ldots \cup F_s'$. By shiftedness, we may assume that $T'=[2,sk+1]$. Recall that
\[
\hf_1(1)=\{F\setminus\{1\}\colon F\in \hf,\ F\cap [q]=\{1\}\}.
\]
Since each member in $\hf_1(1)$ intersects $T'\setminus [q]$, we have
\begin{align}\label{eq3-1}
|\hf_1(1)|\leq (sk+1-q)\binom{n-q-1}{k-2}.
\end{align}

By Claim \ref{claim-4}, we have $\nu(\hf_1)\leq p$. By shiftedness, it follows that $\nu(\hf_1(p+1))\leq p$. Note that $n-q\geq (p+1)(k-1)$. Then Lemma \ref{gbEMC} implies that
\begin{align}\label{eq3-4}
|\hf_1(p+1)| \leq p\binom{n-q-1}{k-2}.
\end{align}

By shiftedness, we have
\[
|\hf_1| \leq \sum_{i=1}^q |\hf_1(i)|\leq p |\hf_1(1)|+ (q-p)|\hf_1(p+1)|.
\]
By \eqref{eq3-1} and \eqref{eq3-4}, it follows that
\begin{align}\label{eqn3-10}
|\hf_1| &\leq p (sk+1-q)\binom{n-q-1}{k-2} + (q-p)p{n-q-1\choose k-2}\nonumber\\[5pt]
& =p (sk+1-p)\binom{n-q-1}{k-2}.
\end{align}
Since $q\geq p$, $p\geq 1$ and $n\geq 8k^2s$, from \eqref{eqn3-10} we have
\begin{align*}
|\hf_1| \leq p sk\binom{n-p-1}{k-2} =  \frac{ 8sk(k-1)}{n-p}\cdot\frac{1}{8} p\binom{n-p}{k-1}\leq  \frac{1}{8} p\binom{n-p}{k-1}.
\end{align*}
\end{proof}

\begin{claim}\label{claim3-2}
\begin{align}\label{eqn3-2}
|\hf_2| \leq  \frac{1}{2} p\binom{n-p}{k-1}.
\end{align}
\end{claim}

\begin{proof}
Note that
\begin{align*}
|\hf_2| = \sum_{P\in \binom{[q]}{2}} |\hf_2(P)|.
\end{align*}
Let $h=\left\lceil\frac{(k-1)(p+1)}{k-2}\right\rceil-1$. Define a graph $\hg \subset \binom{[q]}{2}$ in the following way. The pair $P$ is an edge of $\hg$
iff
\[
|\hf_2(P)|>h\binom{n-q-1}{k-3}.
\]
We claim that $\nu(\hg)\leq h$. Otherwise choose a matching $P_1,\ldots,P_{h+1}$ in $\hg$. Applying Lemma \ref{rainbow}, we can extend it to a matching of size $h+1$ contradicting Claim \ref{claim-4}.

If $q\geq 2h+2$, by Erd\H{o}s matching conjecture for $k=2$ (Erd\H{o}s-Gallai Theorem, \cite{EG})
\[
|\hg| \leq \max\left\{\binom{2h+1}{2},\binom{q}{2}-\binom{q-h}{2}\right\}\leq qh.
\]
If $q\leq 2h+1$, then
\[
|\hg|\leq \binom{q}{2} \leq qh.
\]
Thus
\begin{align}\label{eqn3-11}
|\hg| \leq qh.
\end{align}
Now
\begin{align*}
|\hf_2| &\leq |\hg|\binom{n-q}{k-2} +\left(\binom{q}{2}-|\hg|\right)h\binom{n-q-1}{k-3}\\[5pt]
&=\binom{q}{2}h\binom{n-q-1}{k-3}+\left(\binom{n-q}{k-2}-h\binom{n-q-1}{k-3}\right)|\hg|.
\end{align*}
Since $n > q+(k-1)(p+1)\geq q+(k-2)h$ implies
\[
\binom{n-q}{k-2}-h\binom{n-q-1}{k-3}>0,
\]
from \eqref{eqn3-11} it follows that
\begin{align*}
|\hf_2| &\leq \binom{q}{2}h\binom{n-q-1}{k-3}+ qh\binom{n-q}{k-2}-qh^2\binom{n-q-1}{k-3} \\[5pt]
&\leq \binom{q}{2}h\binom{n-q-1}{k-3}+ qh\binom{n-q}{k-2}\\[5pt]
&\leq h\left(q+\frac{q^2}{2}\cdot \frac{k-2}{n-q} \right)\binom{n-q}{k-2}.
\end{align*}
Since $q\geq p+1$ and $h=\left\lceil\frac{(k-1)(p+1)}{k-2}\right\rceil-1\leq 2p+1$, we have
\begin{align}
|\hf_2| &\leq (2p+1)\left(q+\frac{q^2}{2}\cdot \frac{k-2}{n-q} \right)\binom{n-p-1}{k-2}\\[5pt]
&\leq 3p\left(q+\frac{q^2}{2}\cdot \frac{k-2}{n-q} \right)\cdot\frac{k-1}{n-p}\cdot\binom{n-p}{k-1}\\[5pt]
& =\frac{p}{2}\left(6q \cdot \frac{k-1}{n-p}+3q^2\cdot \frac{(k-1)(k-2)}{(n-p)(n-q)} \right)\binom{n-p}{k-1}.
\end{align}
Since $n\geq 8k^2s$, we have
\[
6q \cdot \frac{k-1}{n-p} \leq \frac{15}{16} \mbox{ and }\ 3q^2\cdot \frac{(k-1)(k-2)}{(n-p)(n-q)} \leq \frac{1}{16}.
\]
Thus \eqref{eqn3-2} follows.
\end{proof}

\begin{claim}\label{claim3-3}
For $i=3,\ldots,k-1$,
\begin{align}\label{eqn3-3}
|\hf_i| \leq \frac{1}{2^i} p\binom{n-p}{k-1}.
\end{align}
\end{claim}

\begin{proof}
 Note that
\begin{align*}
|\hf_i| = \sum_{T\in \binom{[q]}{i}} |\hf_i(T)|.
\end{align*}
Let $h=\left\lceil\frac{(k-1)(p+1)}{k-i}\right\rceil-1$. Define an $i$-graph $\hh \subset \binom{[q]}{i}$ in the following way. A set $T\in \binom{[q]}{i}$ is an edge of $\hh$
iff
\[
|\hf_i(T)|>h\binom{n-q-1}{k-i-1}.
\]
We claim that $\nu(\hh)\leq h$. Otherwise choose a matching $T_1,\ldots,T_{h+1}$ in $\hh$. Applying Lemma \ref{rainbow}, we can extend it to a matching of size $h+1$ contradicting Claim \ref{claim-4}.

If $q\geq (h+1)i$, by Lemma \ref{gbEMC}
\[
|\hh| \leq h \binom{q-1}{i-1}.
\]
If $q< (h+1)i$, then
\[
|\hh|\leq \binom{q}{i} \leq  (h+1) \binom{q-1}{i-1}\leq 2h \binom{q-1}{i-1}.
\]
Thus,
\begin{align}\label{eqn3-12}
|\hh|\leq 2h \binom{q-1}{i-1}.
\end{align}
Now
\begin{align*}
|\hf_i| &\leq |\hh|\binom{n-q}{k-i} +\left(\binom{q}{i}-|\hh|\right)h\binom{n-q-1}{k-i-1}\\[5pt]
&=\binom{q}{i}h\binom{n-q-1}{k-i-1}+\left(\binom{n-q}{k-i}-h\binom{n-q-1}{k-i-1}\right)|\hh|.
\end{align*}
Since $n > q+(k-1)(p+1)\geq q+(k-i)h$ implies
\[
\binom{n-q}{k-i}-h\binom{n-q-1}{k-i-1}>0,
\]
from \eqref{eqn3-12} it follows that
\begin{align}\label{eqn3-14}
|\hf_i| &\leq \binom{q}{i}h\binom{n-q-1}{k-i-1}+\left(\binom{n-q}{k-i}-h\binom{n-q-1}{k-i-1}\right)2h\binom{q-1}{i-1}\nonumber\\[5pt]
&\leq h\binom{q}{i}\binom{n-q-1}{k-i-1}+2h\binom{q-1}{i-1}\binom{n-q}{k-i}.
\end{align}
Note that
\begin{align}\label{eqnn3-2}
h\binom{q}{i}\binom{n-q-1}{k-i-1}&\leq h \frac{q^i}{i!} \frac{(k-i)\cdots (k-1)}{(n-q)\cdots(n-q-1+i)} \binom{n-q-1+i}{k-1}\nonumber\\[5pt]
&\leq \frac{h}{i!} \left(\frac{q(k-1)}{n-q}\right)^i \binom{n-q-1+i}{k-1}
\end{align}
and
\begin{align}\label{eqn3-4}
2h\binom{q-1}{i-1}\binom{n-q}{k-i}&\leq 2h \frac{(q-1)^{i-1}}{(i-1)!}\frac{(k-i+1)\cdots (k-1)}{(n-q+1)\cdots(n-q-1+i)} \binom{n-q-1+i}{k-1}\nonumber\\[5pt]
&\leq  \frac{2h}{(i-1)!}\left(\frac{q(k-1)}{n-q}\right)^{i-1} \binom{n-q-1+i}{k-1}.
\end{align}
Since $n\geq k^2s$ implies $\frac{q(k-1)}{n-q}<1$, from \eqref{eqnn3-2} we have
\begin{align}\label{eqn3-5}
h\binom{q}{i}\binom{n-q-1}{k-i-1} \leq \frac{h}{i!} \left(\frac{q(k-1)}{n-q}\right)^{i-1} \binom{n-q-1+i}{k-1}.
\end{align}
By \eqref{eqn3-4} and \eqref{eqn3-5}, we get
\begin{align*}
|\hf_i| \leq \left(2+\frac{1}{i}\right)\frac{h}{(i-1)!} \left(\frac{q(k-1)}{n-q}\right)^{i-1} \binom{n-q-1+i}{k-1}.
\end{align*}
Recall that $q\geq p+k-1 \geq p+i-1$, we have
\begin{align}\label{eqn3-13}
|\hf_i| \leq \left(2+\frac{1}{i}\right)\frac{h}{(i-1)!} \left(\frac{q(k-1)}{n-q}\right)^{i-1} \binom{n-p}{k-1}.
\end{align}
Since  $h<\frac{(k-1)(p+1)}{k-i}$, we have
\begin{align*}
|\hf_i|&\leq \frac{(2i+1)(k-1)(p+1)}{i(k-i)(i-1)!} \left(\frac{q(k-1)}{n-q}\right)^{i-1} \binom{n-p}{k-1}\\[5pt]
&= \frac{2i+1}{(i-1)!}\cdot\frac{(k-1)(p+1)}{i(k-i)} \left(\frac{q(k-1)}{n-q}\right)^{i-1} \binom{n-p}{k-1}.
\end{align*}
Since $3\leq i\leq k-1$, we have $i(k-i)\geq k-1$. Moreover,
\[
\frac{2i+1}{(i-1)!} \leq \frac{2i+1}{i-1}=2+\frac{3}{i-1}\leq \frac{7}{2}.
\]
 It follows that
\begin{align*}
|\hf_i|&\leq \frac{7}{2}\cdot(p+1) \left(\frac{q(k-1)}{n-q}\right)^{i-1} \binom{n-p}{k-1}\\[5pt]
&\leq 7p \left(\frac{q(k-1)}{n-q}\right)^{i-1} \binom{n-p}{k-1}.
\end{align*}
As $i\geq 3$, it follows that
\begin{align*}
|\hf_i|\leq \frac{1}{2}p\left(\frac{\sqrt{14}q(k-1)}{n-q}\right)^{i-1} \binom{n-p}{k-1}= \frac{1}{2^i} p\left(\frac{2\sqrt{14}q(k-1)}{n-q}\right)^{i-1}\binom{n-p}{k-1}.
\end{align*}
Moreover, $n\geq 8k^2s$ implies $\frac{2\sqrt{14}q(k-1)}{n-q}<1$. Thus \eqref{eqn3-3} follows.
\end{proof}

We also claim that
\begin{align}\label{eqn3-6}
|\hf_k| \leq  \frac{1}{8} p \binom{n-p}{k-1}+\binom{q-p}{k}.
\end{align}
Since $n\geq 8k^2s > 9q$ and $q\geq p+k-1$ implies $n-p-k+1\geq n-q>8q$, it follows that
\begin{align*}
\binom{q-1}{k-1} - \frac{1}{8}  \binom{n-p}{k-1} &\leq \frac{q^{k-1}}{(k-1)!}-\frac{1}{8} \cdot \frac{(n-p-k+1)^{k-1}}{(k-1)!}\\
&\leq \frac{1}{8(k-1)!} \left(8q^{k-1}-(n-p-k+1)^{k-1}\right)\\
&<0.
\end{align*}
Thus,
\begin{align*}
|\hf_k| =\binom{q}{k} &= \frac{1}{8} p \binom{n-p}{k-1}+\binom{q-p}{k}+\left(\binom{q}{k}-\binom{q-p}{k}-\frac{1}{8} p \binom{n-p}{k-1}\right)\\[5pt]
&\leq  \frac{1}{8} p \binom{n-p}{k-1}+\binom{q-p}{k}+\left(p\binom{q-1}{k-1}-\frac{1}{8} p \binom{n-p}{k-1}\right)\\
&<  \frac{1}{8} p \binom{n-p}{k-1}+\binom{q-p}{k}.
\end{align*}

Consequently, by \eqref{eqn3-1}, \eqref{eqn3-2}, \eqref{eqn3-3} and \eqref{eqn3-6} we have
\begin{align*}
|\hf| = \sum_{i=1}^k |\hf_i|
&\leq  \frac{1}{8} p\binom{n-p}{k-1} + \frac{1}{2} p\binom{n-p}{k-1}  +\sum_{i=3}^{k-1}\frac{1}{2^i} p\binom{n-p}{k-1}+\frac{1}{8} p \binom{n-p}{k-1}+\binom{q-p}{k}\\
&= \left(\sum_{i=1}^{k-1}\frac{1}{2^i}\right) p\binom{n-p}{k-1} +\binom{q-p}{k}\\
&< p \binom{n-p}{k-1} +\binom{q-p}{k}\\
&\leq m^*(n,q,k,s).
\end{align*}
This contradicts the assumption that $|\hf|= m^*(n,q,k,s)$. Thus $\hf$ is reducible.
\end{proof}

Now we are ready to prove Theorem \ref{main-1}.
\begin{proof}[Proof of Theorem \ref{main-1}]
By Lemma \ref{lem3.1} and Proposition \ref{prop-2}, the theorem holds for $sk+1\leq q\leq  sk+k-1$. Hence we may assume that $s+k\leq q\leq sk$. Let $p$ be the integer satisfying $(s-p)k+p+1\leq q\leq (s-p)k+p+k-1$.
 By Lemma \ref{reducible}, we see that $\hf,\hf(\bar{1}),\ldots,\hf(\overline{[p-1]})$ are all reducible. Then apply Lemma \ref{recur} $p$ times, we obtain that
\[
m^*(n,q,k,s) = m^*(n-p,q-p,k,s-p) +\sum_{i=1}^p\binom{n-i}{k-1}.
\]
Assume that $q=(s-p)k+p+r$ with $1\leq r\leq k-1$, then $q-p =(s-p)k+r$. By Lemma \ref{lem3.1} for $r<k-1$ and Proposition \ref{prop-2} for $r=k-1$, we have
\[
m^*(n-p,q-p,k,s-p) = \binom{q-p}{k}+\sum_{i=r+1}^{k-1} \binom{q-p-1}{i-1}\binom{n-q}{k-i}.
\]
Thus,
\begin{align*}
m^*(n,q,k,s) &= \binom{q-p}{k}+\sum_{i=r+1}^{k-1} \binom{q-p-1}{i-1}\binom{n-q}{k-i}+\sum_{i=1}^p\binom{n-i}{k-1}\\
&= \binom{n}{k}-\binom{n-p}{k}+\binom{q-p}{k}+\sum_{i=r+1}^{k-1} \binom{q-p-1}{i-1}\binom{n-q}{k-i}.
\end{align*}
This completes the proof.
\end{proof}
\section{The extremal number for general hypergraphs}

Recall that shifting does not increase the matching number and does not decrease the clique number, and $m(n,q,k,s)$ is the maximum value of $m^*(n,t,k,s)$ over all $t\geq q$. In this section, we show that $m^*(n,t,k,s) \geq m^*(n,t+1,k,s)$, which leads to Theorem \ref{main-2}.
\begin{prop}\label{prop-3}
For $s+k-1\leq q\leq sk+k-2$ and $n\geq 2q$,
\[
|\ha(n,q,k,s)| \geq |\ha(n,q+1,k,s)|.
\]
\end{prop}

\begin{proof}
Let $p$ be the integer satisfying $(s-p)k+p+1\leq q\leq (s-p)k+p+k-1$ and  $r=q-p-(s-p)k$.  If $r\leq k-2$, then
\[
|\ha(n,q,k,s)| = \binom{n}{k}-\binom{n-p}{k}+\binom{q-p}{k}+\sum_{i=r+1}^{k-1} \binom{q-p-1}{i-1}\binom{n-q}{k-i}.
\]
and
\[
|\ha(n,q+1,k,s)| =\binom{n}{k}-\binom{n-p}{k}+\binom{q+1-p}{k}+\sum_{i=r+2}^{k-1} \binom{q-p}{i-1}\binom{n-q-1}{k-i}.
\]

Since
\begin{align*}
 \sum_{i=r+1}^{k-1} \binom{q-p-1}{i-1}\binom{n-q}{k-i}= \sum_{i=r+1}^{k-1}\left( \binom{q-p-1}{i-1}\binom{n-q-1}{k-i}+\binom{q-p-1}{i-1}\binom{n-q-1}{k-i-1}\right)
\end{align*}
and
\begin{align*}
\sum_{i=r+2}^{k-1} \binom{q-p}{i-1}\binom{n-q-1}{k-i}= \sum_{i=r+2}^{k-1}\left( \binom{q-p-1}{i-1}\binom{n-q-1}{k-i}+\binom{q-p-1}{i-2}\binom{n-q-1}{k-i}\right),
\end{align*}
it follows that
\begin{align*}
&|\ha(n,q,k,s)|-|\ha(n,q+1,k,s)| \\[5pt]
=&-\binom{q-p}{k-1}+\binom{q-p-1}{r}\binom{n-q-1}{k-r-1}+\binom{q-p-1}{k-2}\\[5pt]
=&\binom{q-p-1}{r}\binom{n-q-1}{k-r-1}-\binom{q-p-1}{k-1}.
\end{align*}
Since $n\geq 2q$, it follows that
\[
|\ha(n,q,k,s)|-|\ha(n,q+1,k,s)| \geq \binom{q-p-1}{r}\binom{q-p-1}{k-r-1}-\binom{q-p-1}{k-1}> 0.
\]

 If $r= k-1$, then
\[
|\ha(n,q,k,s)| = \binom{n}{k}-\binom{n-p}{k}+\binom{q-p}{k}
\]
and
\[
|\ha(n,q+1,k,s)| =\binom{n}{k}-\binom{n-p+1}{k}+\binom{q+2-p}{k}+\sum_{i=2}^{k-1} \binom{q-p+1}{i-1}\binom{n-q-1}{k-i}.
\]
Then,
\begin{align*}
&|\ha(n,q,k,s)|-|\ha(n,q+1,k,s)| \\[5pt]
=& \binom{n-p}{k-1}-\binom{q+1-p}{k-1}-\binom{q-p}{k-1}- \sum_{i=2}^{k-1} \binom{q-p+1}{i-1}\binom{n-q-1}{k-i}\\[5pt]
=& \binom{n-p}{k-1}-\binom{q+1-p}{k-1}-\binom{q-p}{k-1}- \binom{n-p}{k-1} +\binom{n-q-1}{k-1}+\binom{q-p-1}{k-1}\\[5pt]
= & \binom{n-q-1}{k-1} +\binom{q-p-1}{k-1} -\binom{q+1-p}{k-1}-\binom{q-p}{k-1}\\[5pt]
= & \binom{n-q-1}{k-1} -\binom{q-p-1}{k-2} -\binom{q+1-p}{k-1}.
\end{align*}
For $n\geq  2q$, the first term is at least $\binom{q+1}{k-1}$. For $p\geq 1$, the sum of the negative terms is less than
\[
\binom{q}{k-1} +\binom{q}{k-2} = \binom{q+1}{k-1}
\]
proving that the expression is positive.
\end{proof}

Now we are in position to prove Theorem \ref{main-2}.

\begin{proof}[Proof of Theorem \ref{main-2}]
 Let $\hf$ be a shifted $k$-graph with $\nu(\hf)=s$, $\omega(\hf)\geq q$ and $|\hf|=m(n,q,k,s)$. If $q\leq s+k-1$, then Proposition \ref{prop-2} implies  either $\hf\subset \mathcal{E}(n,k,s)$ or $\omega(\hf)\geq s+k$. By Proposition \ref{prop-3}, it follows that
 \[
 m(n,q,k,s)= \max\left\{|\mathcal{E}(n,k,s)|,\max_{t\geq s+k} m^*(n,t,k,s)\right\} = |\mathcal{E}(n,k,s)|.
 \]

If $s+k\leq q\leq sk+k-1$ and $n\geq 8k^2s$,
then Theorem \ref{main-1} and Proposition \ref{prop-3} imply that
\[
m(n,q,k,s) =\max_{t\geq q}  m^*(n,t,k,s) =  m^*(n,q,k,s).
\]
\end{proof}

\section{The extremal number for large $q$ and small $n$}

In this section, we consider the problem for large $q$ and small $n$, where the extremal number
is attained by the family $\binom{[sk+k-1]}{k}$.

\begin{proof}[Proof of Theorem \ref{main-4}]
Let $\hf\subset \binom{[n]}{k}$ be a $k$-graph with $\nu(\hf)\leq s$ and $\omega(\hf)=q=sk+k-l$. We may assume that $\hf$ is shifted.  By shiftedness, we have $\binom{[q]}{k}\subset \hf$ and $\{q-k+1,q-k+2,\ldots,q+1\}\notin \hf$.
 
\begin{claim}\label{claim5-1}
For any $F\in \hf$,
\[
|F\cap [q+1, n]|\leq \min\{k,l\}-1.
\]
\end{claim} 
\begin{proof}
We first show that $F\cap [q]\neq \emptyset$. Otherwise, if $F\cap [q]=\emptyset$, then $\{q-k+1,q-k+2,\ldots,q+1\}\prec F$ and  $\{q-k+1,q-k+2,\ldots,q+1\}\notin \hf$, which contradicts the fact that $\hf$ is shifted. 

If $l<k-1$, then $|F\cap [q]|\geq k-l+1$. Otherwise, if $|F\cap [q]|\leq k-l$, then $|[q]\setminus F|\geq sk+k-l-(k-l)=sk$. There are $s$ pairwise disjoint sets $F_1,\ldots,F_s$ in $[q]\setminus F$. It follows that $F,F_1,\ldots,F_s$ form a matching of size $s+1$, which contradicts $\nu(\hf)\leq s$.
Hence 
\[
|F\cap [q]|\geq \max\{k-l,0\}+1
\]
and
\[
|F\cap [q+1, n]|=k-|F\cap [q]|\leq \min\{k,l\}-1.
\]
\end{proof}

By Claim \ref{claim5-1}, we have 
\begin{align}\label{eq5-0}
|\hf| =\sum_{T\subset [q+1,n]\atop |T|\leq \min\{k,l\}-1} \hf(T,[q]).
\end{align}
Now we distinguish two cases.

Case 1. For each $i=1,\ldots,l-1$, $|\hf(q+l-i,[q])|>\binom{q}{k-1}-\binom{q-i(k-1)}{k-1}$.

We first claim that $\hf(q+l,[q])=\emptyset$. Suppose not, let $E_l\in \hf(q+l,[q])$. Since $|\hf(q+l-1,[q])-E_l| > 0$, we can choose $E_{l-1}$ from $\hf(q+l-1,[q])-E_l$. Similarity, we can choose $E_{i}$ from $\hf(q+i,[q])-(E_l\cup E_{l-1}\cup \ldots E_{i+1})$ for $i=l-1,\ldots,1$. Then
\[
M=\{E_1\cup\{q+1\}, E_2\cup\{q+2\},\ldots,E_l\cup\{q+l\}\}
\]
is a matching of size $l$ in $\hf$. Since $\binom{[q]}{k}\subset \hf$ and $q-l(k-1)=(s+1-l)k$,  $M$ can be extended into a matching of size $s+1$, contradicting  the fact that $\nu(\hf)\leq s$.  Thus $\hf(q+l,[q])=\emptyset$.

From $\hf(q+l,[q])=\emptyset$ and $\hf$ is shifted, it follows that  $F\cap [q+l,n]= \emptyset$ for any $F\in \hf$. Therefore,
\[
|\hf| \leq \binom{q+l-1}{k} =\binom{(s+1)k-1}{k}.
\]

Case 2. There is an $i\in [1,l-1]$ with
\begin{align}\label{eq5-1}
|\hf(q+l-i,[q])|\leq \binom{q}{k-1}-\binom{q-i(k-1)}{k-1}.
\end{align}

Let $T\in \binom{[q+1,n]}{a}$ with $T\cap [q+l-i,n]\neq \emptyset$. Define
\[
\hb(T) = \{[q]\setminus E \colon E\in \hf(T,[q])\}
\]
and
\[
\hb(q+l-i) =  \{[q]\setminus E \colon E\in \hf(q+l-i,[q])\}.
\]

Let $\partial^{(a-1)} \hb(T)$ be the $(a-1)$-th shadow of $\hb(T)$.
\begin{claim}
$$ \partial^{(a-1)} \hb(T)\subset \hb(q+l-i)$$
\end{claim}

\begin{proof}
Let $B_0 \in \partial^{(a-1)} \hb(T)$. Then there is a $B\in \hb(T)$ such that $B_0\subset B$.
By the definition of $\hb(T)$, we have $[q]\setminus B \in \hf(T,[q])$. Note that $T\cap [q+l-i,n]\neq \emptyset$.
There exists $x\in T$ with $x\geq q+l-i$. Since $\hf$ is shifted and $(B\setminus B_0) \prec (T\setminus \{x\}) $,
we have $([q]\setminus B) \cup (B\setminus B_0) = [q]\setminus B_0 \in \hf(x,[q]) \subset \hf(q+l-i,[q])$.
\end{proof}

Now we claim that
\begin{align}\label{eq5-2}
|\hf(T,[q])|= |\hb(T)|\leq \binom{q}{k-a}-\binom{q-i(k-1)}{k-a} = \sum_{j=1}^{i(k-1)} \binom{q-j}{q-j-k+a+1}.
\end{align}
Otherwise, by Kruskal-Katona Theorem,
\[
 \partial^{(a-1)} \hb(T) > \sum_{j=1}^{i(k-1)} \binom{q-j}{q-j-k+2} =\sum_{j=1}^{i(k-1)} \binom{q-j}{k-2},
\]
which contradicts the condition that
\[
|\hf(q+l-i,[q])|\leq \binom{q}{k-1}-\binom{q-i(k-1)}{k-1}.
\]

Now by shiftedness, we have
\[
\sum_{j=q+l-i}^n |\hf(j,[q])| \leq (n-q) |\hf(q+l-i,[q])|.
\]
By \eqref{eq5-1}, it follows that
\begin{align*}
\sum_{j=q+l-i}^n |\hf(j,[q])| &\leq (n-q)\left(\binom{q}{k-1}-\binom{q-i(k-1)}{k-1}\right)\\[5pt]
&\leq (n-q) i(k-1) \binom{q-1}{k-2}\\[5pt]
&\leq \frac{(n-q)i(k-1)^2}{q}\binom{q}{k-1}.
\end{align*}
Since $q=sk+k-l \geq sk+k-\frac{s}{3k} > (k-1)s$ and $n\leq q+\frac{s}{3k}$,
\[
\frac{(n-q)(k-1)^2}{q} \leq \frac{1}{2}.
\]
Therefore,
\begin{align}\label{eqn5-1}
\sum_{j=q+l-i}^n |\hf(j,[q])| \leq \frac{i}{2}\binom{q}{k-1}.
\end{align}

For each $a=2,\ldots, \min\{k,l\}-1$, by \eqref{eq5-2} we have
\begin{align*}
\sum_{T\in \binom{[q+1,n]}{a}\atop T\cap [q+l-i,n]\neq \emptyset} |\hf(T,[q])| &\leq \binom{n-q}{a} \left( \binom{q}{k-a}-\binom{q-i(k-1)}{k-a}\right)\\[5pt]
&\leq i(k-1)\binom{n-q}{a} \binom{q-1}{k-a-1}.
\end{align*}
Since $\binom{n}{k}\leq \left(\frac{en}{k}\right)^k$, $q\geq k+(k-1)s$ and  $n\leq q+\frac{s}{3k}$, it follows that
\begin{align}\label{eqn5-2}
\sum_{T\in \binom{[q+1,n]}{a}\atop T\cap [q+l-i,n]\neq \emptyset} |\hf(T,[q])| & \leq i(k-1)\left(\frac{e(n-q)}{a}\right)^a \cdot\frac{(k-1)(k-2)\cdots(k-a)}{q(q-k+a)\cdot(q-k+2)} \cdot \binom{q}{k-1}\nonumber\\[5pt]
& \leq i(k-1)\left(\frac{e(n-q)}{sa}\right)^a \binom{q}{k-1}\nonumber\\[5pt]
&\leq \frac{i}{2^a} \binom{q}{k-1}.
\end{align}

Note that
\begin{align}\label{eq5-3}
\sum_{j=q+1}^{q+l-i-1} |\hf(j,[q])| \leq (l-1-i) \binom{q}{k-1}
\end{align}
and
\begin{align}\label{eq5-4}
\sum_{T\in \binom{[q+1,n]}{a}\atop T\subset [q+1, q+l-i-1]} |\hf(T,[q])| \leq \binom{l-1}{a} \binom{q}{k-a}
\end{align}
for each $a=2,\ldots, \min\{k,l\}-1$.
Moreover, by \eqref{eqn5-1} and \eqref{eqn5-2} we obtain
\begin{align}\label{eq5-5}
\sum_{a=1}^{\min\{k,l\}-1}\sum_{T\in \binom{[q+1,n]}{a}\atop T\cap [q+l-i,n]\neq \emptyset} |\hf(T,[q])| \leq \left(\sum_{a=1}^\infty \frac{1}{2^a} \right) i\binom{q}{k-1} = i\binom{q}{k-1}.
\end{align}
By \eqref{eq5-0}, \eqref{eq5-3}, \eqref{eq5-4} and \eqref{eq5-5}, we have
\begin{align*}
|\mathcal{F}|& = \binom{q}{k}+ \sum_{j=q+1}^{q+l-i-1} |\hf(j,[q])|+ \sum_{a=1}^{\min\{k,l\}-1}\sum_{T\in \binom{[q+1,n]}{a}\atop T\cap [q+l-i,n]\neq \emptyset} |\hf(T,[q])| \\[5pt]
&\qquad\qquad+\sum_{a=2}^{\min\{k,l\}-1}\sum_{T\in \binom{[q+1,n]}{a}\atop T\subset [q+1, q+l-i-1]} |\hf(T,[q])|\\[5pt]
&\leq {q\choose k}+(l-1-i){q\choose k-1}+i{q\choose k-1}+\sum_{a=2}^{\min\{k,l\}-1}{l-1\choose a}{q\choose k-a}\\[5pt]
&={sk+k-1\choose k}.
\end{align*}
\end{proof}

\section{A general cross-intersecting theorem and two special cases}

In this section, we prove a general cross-intersecting theorem and confirm Conjecture \ref{conj-2} for the cases $q=sk+k-2$ and $k=2$.

\begin{proof}[Proof of Theorem \ref{prop6-1}]
First note that $\nu(\mathcal{A})\leq s$ and the Erd\H{o}s matching conjecture for $n\geq 2(s+1)k$  in \cite{F13} imply
\begin{align*}
|\mathcal{A}|\leq{n\choose k}-{n-s\choose k},
\end{align*}
and $n\geq(l-t+1)(t+1)$ and the $t$-intersecting property imply
\begin{align}\label{6-3}
|\mathcal{B}|\leq{n-t\choose l-t}.
\end{align}
These prove \eqref{6-1} for the case $s=t$. Suppose $s>t$. If $|\mathcal{A}|\leq{n\choose k}-{n-t\choose k}$ then \eqref{6-3} implies \eqref{6-1} with $i=t$. Let $|\mathcal{A}|>{n\choose k}-{n-t\choose k}$ and define $j$ by
\begin{align}\label{6-4}
{n\choose k}-{n-j\choose k}<|\mathcal{A}|\leq{n\choose k}-{n-j-1\choose k},
\end{align}
and note $t\leq j\leq s-1$. Let $\widetilde{\mathcal{A}}$ ($\widetilde{\mathcal{B}}$) denote the first $|\mathcal{A}|$ ($|\mathcal{B}|$) elements of ${[n]\choose k}$ (${[n]\choose l}$), respectively in the lexicographic order. In view of $n\geq k+l$ and the Kruskal-Katona Theorem (or Hilton's Lemma, see \cite{FK17,Hilton76}) $\widetilde{\mathcal{A}}$ and $\widetilde{\mathcal{B}}$ are cross-intersecting. By the definition of the lexicographic order and \eqref{6-4}, $A\cap[j+1]\neq\emptyset$ for all $A\in\widetilde{\mathcal{A}}$ which guarantees $\nu(\widetilde{\mathcal{A}})\leq j+1\
\leq s$. Also, the first part of \eqref{6-4} guarantees that all $k$-sets $A\in{[n]\choose k}$ satisfying $A\cap[j]\neq\emptyset$ are in $\widetilde{\mathcal{A}}$. This in turn via the cross-intersecting property implies
\begin{align}\label{6-5}
[t]\subset[j]\subset B \text{ for every } B\in\widetilde{\mathcal{B}}.
\end{align}
In particular, $\widetilde{\mathcal{B}}$ is $t$-intersecting. Thus $\widetilde{\mathcal{A}}$, $\widetilde{\mathcal{B}}$ satisfy the conditions imposed on $\mathcal{A}$ and $\mathcal{B}$.

Let $V\subset [n]$. Recall the definition:
\begin{align*}
\mathcal{F}(\overline{V})=\{F\in\mathcal{F}: F\cap V=\emptyset\},\qquad \mathcal{F}(V)=\{F\setminus V: V\subset F\in\mathcal{F}\}.
\end{align*}
Then
\[|\widetilde{\mathcal{A}}(\overline{[j]})|=|\widetilde{\mathcal{A}}|-\left({n\choose k}-{n-j\choose k}\right)\]
and
$(j+1)\in A$ for all $A\in\widetilde{\mathcal{A}}(\overline{[j]})$. Also $|\widetilde{\mathcal{B}}([j+1])|={n-j-1\choose l-j-1}$. Define $\mathcal{B}_0=\{B\setminus [j]: B\in\widetilde{\mathcal{B}}, j+1\notin B\}$. In view of \eqref{6-5}
\[\mathcal{B}_0\subset{[j+2,n]\choose l-j} \text{ and } |\mathcal{B}_0|=|\widetilde{\mathcal{B}}|-{n-j-1\choose l-j-1}.\]
Set
\[\mathcal{A}_0=\{A\setminus\{j+1\}: A\in \widetilde{\mathcal{A}}(\overline{[j]})\}\subset{j+2, n\choose k-1},\]
$|\mathcal{A}_0|=|\widetilde{\mathcal{A}}|-\left({n\choose k}-{n-j\choose k}\right)$.
\begin{claim}
\[|\mathcal{A}_0|+\beta|\mathcal{B}_0|\leq \max\left\{{n-j-1\choose k-1}, \beta{n-j-1\choose l-j}\right\}.\]
\end{claim}
\begin{proof}
Since $\widetilde{\mathcal{A}}$, $\widetilde{\mathcal{B}}$ are cross-intersecting, $\mathcal{A}_0$, $\mathcal{B}_0$ are cross-intersecting as well. Consider the bipartite graph $G$ with vertex-set $\left({[j+2,n]\choose k-1}, {[j+2, n]\choose l-j}\right)$ and the two vertices $E, F$ forming an edge iff $E\cap F=\emptyset$. Then $G$ is a regular graph and $\mathcal{A}_0\cap\mathcal{B}_0$ is an independent set. This easily implies
\[|\mathcal{A}_0|+\beta|\mathcal{B}_0|\leq \max\left\{{n-j-1\choose k-1}, \beta{n-j-1\choose l-j}\right\}.\]
\end{proof}
To derive \eqref{6-1} from the claim is easy. If the first term gives the maximum then
\begin{align*}
|\widetilde{\mathcal{A}}|+\beta|\widetilde{\mathcal{B}}|=&{n\choose k}-{n-j\choose k}+\beta{n-j-1\choose l-j-1}+|\mathcal{A}_0|+\beta|\mathcal{B}_0|\\[5pt]
\leq&{n\choose k}-{n-j-1\choose k}+\beta{n-j-1\choose l-j-1}.
\end{align*}
If the second term gives the maximum we obtain in the same way
\[|\widetilde{\mathcal{A}}|+\beta|\widetilde{\mathcal{B}}|\leq {n\choose k}-{n-j\choose k}+\beta{n-j-1\choose l-j-1}+\beta{n-j-1\choose l-j}={n\choose k}-{n-j\choose k}+\beta{n-j\choose l-j}.\]
\end{proof}

\begin{proof}[Proof of Theorem \ref{specialcase-2}]
We may assume that $\hf$ is shifted. If $\omega(\hf) = sk+k-1$, then Proposition \ref{prop-2} implies that $|\hf|=\binom{sk+k-1}{k}$. Hence, we may further assume that $\omega(\hf) = sk+k-2$. Then $\mathcal{F}={[q]\choose k}\cup\mathcal{F}_{k-1}$. Recall that for each $i\in [q+1,n]$,
\[
\hf(i,[q])=\left\{E\in{[q]\choose k-1}: E\cup\{i\} \in\mathcal{F}\right\}.
\]
 Let $\ha = \hf(q+2,[q])$ and $\hb=\hf(q+1,[q])$. The following simple result is straightforward.
\begin{claim}
$\mathcal{A}$ is intersecting, $\mathcal{A}$ and $\mathcal{B}$ are cross-intersecting and
\begin{align*}
|\mathcal{F}_{k-1}|=\sum_{i=q+1}^n |\hf(i,[q])| \leq (n-q-1)|\mathcal{A}|+|\mathcal{B}|.
\end{align*}
\end{claim}

Note that $q=sk+k-2\geq 3(k-1)$ for $s\geq 2$. By Theorem \ref{prop6-1}, we have
\begin{align*}
|\hf_{k-1}| &\leq(n-q-1)\left(|\mathcal{A}|+\frac{1}{n-q-1}|\mathcal{B}|\right)\\[5pt]
& \leq (n-q-1) \max\left\{\frac{1}{n-q-1}{q\choose k-1}, {q-1\choose k-2}+\frac{1}{n-q-1}{q-1\choose k-2}\right\}\\[5pt]
&= \max\left\{{q\choose k-1}, (n-q){q-1\choose k-2}\right\}.
\end{align*}
Thus, for all $n\geq (s+1)k$, we have
\[|\mathcal{F}|\leq \max\left\{{sk+k-1\choose k}, {sk+k-2\choose k}+{sk+k-3\choose k-2}(n-q)\right\}.\]
\end{proof}

At the end, we confirm Conjecture \ref{conj-2} for $k=2$.

\begin{proof}[Proof of Theorem \ref{specialcase-1}]
Let $\hg$ be a graph with $\nu(\hg) \leq s$ and $\omega(\hg)\geq q$. Without loss of generality we assume that $\hg$ is shifted. Clearly, we have $\binom{[q]}{2}\subset \hg$.
Consider the following matching $$E_i =(i,2s+3-i),\ 1\leq i\leq s+1.$$ Since $\nu(\hg) \leq s$, there exists some $j$ with $1\leq j\leq s+1$  such that $E_j\notin \hg$. Moreover, $\binom{[q]}{2}\subset \hg$ implies $2s+3-j>q$, i.e., $j\leq 2s+2-q$. Let $\hg_j$ be the graph with
\[
E(\hg_j) =\{(i,x)\colon 1\leq i<j,\ i<x\leq n\} \cup \binom{[2s+2-j]}{2}.
\]
From Remark 2 of \cite{AF85}, $\hg$ is a subgraph of $\hg_j$. Hence, we have
\[
|\hg| \leq |\hg_j| = \binom{2s+2-j}{2} +(j-1)(n-2s-2+j).
\]
Since $1\leq j\leq 2s+2-q$, we get
\[
|\hg| \leq \max\left\{\binom{2s+1}{2},\binom{q}{2}+(2s+1-q)(n-q)\right\}.
\]
This proves the theorem.
\end{proof}

\end{document}